\documentclass[12pt]{amsart}
\usepackage{a4wide, amssymb}

\begin{document}

\newcommand{\iii}{{\mathrm i}}
\newcommand{\Ht}{H}
\newcommand{\vs}{{\mathfrak a}}
\newcommand{\Ka}{K_\infty^0}
\newcommand{\modulus}{\delta}
\newcommand{\spclrt}{\delta}
\newcommand{\cu}{\operatorname{cus}}
\newcommand{\der}{\operatorname{der}}
\newcommand{\op}[1]{\overline{#1}}
\newcommand{\di}{\operatorname{dis}}
\newcommand{\rts}{\Phi}
\newcommand{\End}{\operatorname{End}}
\newcommand{\Mat}{\operatorname{Mat}}
\newcommand{\nonorth}{\mathcal{N}}
\newcommand{\GSpin}{\operatorname{GSpin}}
\newcommand{\F}{\mathbb{F}}
\newcommand{\nm}{{\mathcal N}}
\newcommand{\C}{{\mathbb C}}
\newcommand{\N}{{\mathbb N}}
\newcommand{\R}{{\mathbb R}}
\newcommand{\Z}{{\mathbb Z}}
\newcommand{\Q}{{\mathbb Q}}
\newcommand{\grp}[1]{{\bf{#1}}}
\newcommand{\Id}{\operatorname{Id}}
\newcommand{\A}{{\mathbb A}}
\newcommand{\bH}{{\mathbb H}}
\newcommand{\Lie}{\operatorname{Lie}}
\newcommand{\OOO}{\mathcal{O}}
\newcommand{\varbeta}{\tilde\beta}
\newcommand{\srts}{\Delta}
\newcommand{\gf}{{\mathfrak g}}
\newcommand{\mf}{{\mathfrak m}}
\newcommand{\kf}{{\mathfrak k}}
\newcommand{\ho}{{\mathfrak o}}
\newcommand{\pg}{{\mathfrak p}}
\newcommand{\mM}{{\mathfrak M}}
\newcommand{\mN}{{\mathfrak N}}
\newcommand{\HHH}{{\mathcal H}}
\newcommand{\M}{{\mathcal M}}
\newcommand{\Latt}{{\mathcal Z}}
\newcommand{\Co}{{\mathcal C}}
\newcommand{\cO}{{\mathcal O}}
\newcommand{\pars}{{\mathcal P}}
\newcommand{\leviP}{{\mathcal F}}
\newcommand{\levis}{{\mathcal L}}
\newcommand{\cA}{{\mathcal A}}
\newcommand{\AF}{{\mathcal A}}
\newcommand{\K}{\mathbf{K}}
\newcommand{\LieU}{\mathfrak{u}}
\newcommand{\LieM}{\mathfrak{m}}
\newcommand{\LieUbar}{\op{\mathfrak{u}}}
\newcommand{\LieG}{\mathfrak{g}}
\newcommand{\cU}{{\mathcal U}}
\newcommand{\plnch}{\operatorname{plnch}}
\newcommand{\bs}{\backslash}
\newcommand{\sm}[4]{\left(\begin{smallmatrix}{#1}&{#2}\\{#3}&{#4}\end{smallmatrix}\right)}
\newcommand{\temp}{\operatorname{temp}}
\newcommand{\disc}{\operatorname{disc}}
\newcommand{\cusp}{\operatorname{cusp}}
\newcommand{\sprod}[2]{\left\langle#1,#2\right\rangle}
\renewcommand{\Im}{\operatorname{Im}}
\renewcommand{\Re}{\operatorname{Re}}
\newcommand{\Ind}{\operatorname{Ind}}
\newcommand{\Tr}{\operatorname{Tr}}
\newcommand{\tr}{\operatorname{tr}}
\newcommand{\rad}{\operatorname{rad}}
\newcommand{\Hom}{\operatorname{Hom}}
\newcommand{\Ker}{\operatorname{Ker}}
\newcommand{\vol}{\operatorname{vol}}
\newcommand{\Area}{\operatorname{Area}}
\newcommand{\SL}{\operatorname{SL}}
\newcommand{\GL}{\operatorname{GL}}
\newcommand{\Sp}{\operatorname{Sp}}
\newcommand{\SO}{\operatorname{SO}}
\newcommand{\Ad}{\operatorname{Ad}}
\newcommand{\supp}{\operatorname{supp}}
\newcommand{\nnn}{{\mathfrak n}}
\newcommand{\ppp}{{\mathfrak p}}
\newcommand{\orbit}{{\mathcal O}}
\newcommand{\mmm}{{\mathfrak m}}
\newcommand{\kkk}{{\mathfrak k}}
\newcommand{\card}[1]{\lvert{#1}\rvert}
\newcommand{\abs}[1]{\lvert#1\rvert}
\newcommand{\norm}[1]{\lVert#1\rVert}
\newcommand{\one}{\mathbf 1}
\newcommand{\aaa}{\mathfrak{a}}
\newcommand{\eps}{\epsilon}
\newcommand{\ad}{\operatorname{ad}}
\newcommand{\proj}{\operatorname{proj}}
\newcommand{\level}{\operatorname{level}}

\newtheorem{theorem}{Theorem}
\newtheorem{claim}{Claim}
\newtheorem{lemma}{Lemma}
\newtheorem{conjecture}{Conjecture}
\newtheorem{hypothesis}{Hypothesis}
\newtheorem{proposition}{Proposition}
\newtheorem{corollary}{Corollary}
\newtheorem{result}{Result}
\newtheorem{assumption}{Assumption}
\theoremstyle{definition}
\newtheorem{definition}{Definition}
\theoremstyle{remark}
\newtheorem*{example}{Example}
\newtheorem*{question}{Question}
\newtheorem{remark}{Remark}

\title[Analytic properties of intertwining operators II]{On the analytic properties of intertwining operators II: local degree bounds and limit multiplicities}

\author{Tobias Finis}
\address{Universit\"at Leipzig, Mathematisches Institut, Postfach 100920, D-04009 Leipzig, Germany}
\email{finis@math.uni-leipzig.de}

\author{Erez Lapid}
\address{Department of Mathematics, The Weizmann Institute of Science, Rehovot 7610001, Israel}
\email{erez.m.lapid@gmail.com}

\date{\today}

\begin{abstract}
In this paper we continue to study the degrees of matrix coefficients of intertwining operators associated to reductive groups over $p$-adic local fields.
Together with previous analysis of global normalizing factors, 
we can control the analytic properties of global intertwining operators for a large class of reductive groups over number fields,
in particular for inner forms of $\GL(n)$ and $\SL(n)$ and quasi-split classical groups.
This has a direct application to the limit multiplicity problem for these groups. 
\end{abstract}

\maketitle

\tableofcontents

\section{Introduction}
Let $\grp{G}$ be a reductive algebraic group defined over a $p$-adic field $F$ with residue field $\F_q$ and let $\varpi$ be
a uniformizer of $F$. Let $G = \grp{G} (F)$ and let
$K_0$ be a special maximal compact subgroup of $G$.
Let $\grp{P} = \grp{M} \grp{U}$ be a maximal parabolic subgroup of $\grp{G}$ defined over $F$, $\grp{\op P}$ the opposite parabolic subgroup
and $\pi$ a smooth irreducible representation of $M = \grp{M} (F)$ on a complex vector space.
We consider the family of induced $G$-representations $I_P(\pi,s)$, $s\in\C$, which extend the
fixed $K_0$-representation $I^{K_0}_{P \cap K_0} (\pi|_{M \cap K_0})$,
and the associated
intertwining operators
\[
M (s) = M_{\op P|P}(\pi,s) :
I^{K_0}_{P \cap K_0} (\pi|_{M \cap K_0}) =
 I_P(\pi,s)\rightarrow I_{\op P}(\pi,-s) = I^{K_0}_{\op P \cap K_0} (\pi|_{M \cap K_0}),
\]
which we regard as a family of linear maps between
the vector spaces $I^{K_0}_{P \cap K_0} (\pi|_{M \cap K_0})$ and
$I^{K_0}_{\op P \cap K_0} (\pi|_{M \cap K_0})$ that are independent of $s$.
For any closed subgroup $K$ of $K_0$, let
\[
M(s)^K :
I^{K_0}_{P \cap K_0} (\pi|_{M \cap K_0})^K \to
I^{K_0}_{\op P \cap K_0} (\pi|_{M \cap K_0})^K
\]
be the restriction of $M(s)$ to the space of $K$-invariant vectors
in $I^{K_0}_{P \cap K_0} (\pi|_{M \cap K_0})$.
We recall that the matrix coefficients of the linear operators $M(s)$ are rational functions of $q^{-s}$,
and that the degrees of the denominators are bounded independently of $\pi$ (cf. Remark \ref{RemarkRationality} below).
For any $n \ge 1$ write $K_n$ for the principal congruence subgroup of $K_0$ of level $\varpi^n$ with respect to a fixed faithful
$F$-rational representation $\rho$ of $G$ and a suitable lattice $\Lambda_\rho$ in the space of $\rho$ (see
\eqref{EqnPrincipalCongruence} below for the precise definition).
Let $\grp{\hat M}$ be the $F$-simple normal subgroup of $\grp{G}$ generated by $\grp{U}$ and $\grp{\op U}$.

We are interested in the following property of the group $G$.

\begin{definition}
\begin{enumerate}
\item A group $G$ satisfies property (BDmax), if there exists a constant $C > 0$ such that for all maximal parabolic subgroups $P = MU$,
smooth irreducible representations $\pi$ of $M$ and all $n \ge 1$,
the degrees of the numerators of the matrix coefficients of $M(s)^{K_n \cap \hat{M}}$ are bounded by $C n$.
\item A group satisfies property (BD), if all its Levi subgroups satisfy property (BDmax).
\end{enumerate}
\end{definition}

We also propose
the following supplement in a global situation, where we consider a reductive group $\grp{G}$ defined over a number field $k$ and
its base change to $F = k_v$ for all non-archimedean places $v$ of $k$. In this case,
we obtain the open compact subgroups $K_{0,v}$ and $K_{n,v}$ of $\grp{G} (k_v)$ from a fixed faithful $k$-rational representation $\rho$ of $G$ and
an $\OOO_k$-lattice $\Lambda_\rho$ in the space of $\rho$. The groups $K_{0,v}$
will be automatically hyperspecial for almost all $v$, and we can ensure that they are special for all $v$.


\begin{definition}
We say that $\grp{G}$ satisfies property (BD) for a set $S$ of non-archimedean valuations of $k$, if the local groups $\grp{G} (k_v)$, $v \in S$,
satisfy (BD) with a uniform value of $C$.
\end{definition}

We conjecture that these properties hold for all reductive groups over local fields or number fields, respectively (with $S=S_{\operatorname{fin}}$,
the set of all non-archimedean valuations, in the latter case).

In \cite{MR3001800}, property (BD) was established for the groups $\GL (n)$ and $\SL (n)$ over local fields and number fields
(with $S=S_{\operatorname{fin}}$ in the latter case).\footnote{We remark that the definition of property (BD) for a reductive group $\grp{G}$
over a number field $k$ in \cite{MR3352530} is slightly different from the current definition.
Our current formulation seems more natural, and in any case our current property (BD) for the set $S_{\operatorname{fin}}$ implies property (BD)
in the old formulation (see Remark \ref{RemarkBD} below).}
The main result of the current paper is that property (BD) for a group $G$ is implied by a quantitative bound
on the support of supercuspidal matrix coefficients (property (PSC), see Definition \ref{DefinitionPSC})
for all semisimple normal subgroups of proper Levi subgroups of $G$ (Theorem \ref{TheoremBD} below). Previously, this
implication had been proven in [ibid.] only for a restricted class of parabolic subgroups $\grp{P}$.
By [ibid., Corollary 13], Property (PSC) holds for a reductive group if all its irreducible supercuspidal representations
are induced from cuspidal representations of subgroups that are open compact modulo the center.
As a consequence of Kim's exhaustion theorem for supercuspidal representations \cite{MR2276772}, we conclude that for every
reductive group $\grp{G}$ over a number field $k$ there exists a finite set $S_0$ of non-archimedean places of $k$ such that
$\grp{G}$ has property (BD) for the set $S_{\operatorname{fin}} - S_0$ (Theorem \ref{TheoremGlobal}).
The only remaining question is to establish the local property (BD) for the finitely many $v \in S_0$.
Moreover, using additional known results on supercuspidal representations, for inner forms of $\GL(n)$ or $\SL(n)$
we can take $S_0$ to be empty, while for classical groups we can take it to be the set of all places of residual
characteristic $2$ (Corollary \ref{CorClassical}).
At the end of the paper, we combine our results with \cite{1603.05475} and spell out the consequences for the limit multiplicity
problem for inner forms of $\GL(n)$ or $\SL(n)$, quasi-split classical groups and groups of split rank two (Corollary \ref{CorLM}).

\section{Notation}
Let $F$ be a $p$-adic field (i.e. a non-archimedean local field of characteristic $0$) with normalized absolute value $\abs{\cdot}$
Let $\OOO$ be the ring of integers of $F$, $\varpi$ a uniformizer of $F$ and $q$ the cardinality of the residue field of $F$.

As a rule, we write $X=\grp{X}(F)$ whenever $\grp{X}$ is a variety over $F$.
Let $\grp{G}$ be a connected reductive algebraic group defined over $F$.
All algebraic subgroups that will be considered in the sequel are implicitly assumed to be defined over $F$.
Fix a maximal $F$-split torus $\grp{T}_0$ and a minimal parabolic subgroup $\grp{P}_0 = \grp{M}_0 \grp{U}_0 \supset \grp{T}_0$ of $\grp{G}$,
where $\grp{M}_0 = C_{\grp{G}} (\grp{T}_0)$ is a minimal Levi subgroup of $\grp{G}$. Let $\rts = R(\grp{T}_0,\grp{G})$
be the set of roots of $\grp{T}_0$ and $\Sigma \subset \rts$ the subset of reduced roots. For any algebraic subgroup $\grp{X}$ of $\grp{G}$ normalized by $\grp{T}_0$ write $\rts_{\grp{X}} = R(\grp{T}_0,\grp{X})\subset \rts$.
For any
$\alpha \in \rts$ let $\mathfrak{u}_\alpha$ be its root space in the Lie algebra
$\mathfrak{g} = \Lie \grp{G}$ of $\grp{G}$, and
for any
$\alpha \in \Sigma$ let $\grp{U}_\alpha$ be the associated unipotent subgroup of $\grp{G}$ (i.e. the subgroup whose Lie algebra is the sum of
the root spaces $\mathfrak{u}_{i \alpha}$, $i > 0$).
For any root $\alpha \in \rts$
its absolute value $\abs{\alpha}$ on $T_0$ extends uniquely to a homomorphism $\abs{\alpha}: M_0 \rightarrow\R_{>0}$.

The choice of $\grp{P}_0$ fixes a set of positive roots $R(\grp{T}_0,\grp{U}_0) \subset \rts$.
Let $\Delta_0 \subset \Sigma$ be the corresponding subset of simple roots.
For any standard parabolic subgroup $\grp P$ of $\grp{G}$ with standard Levi decomposition $\grp{P}=\grp{M}\grp{U}$ we denote by
$\grp{\op P}=\grp{M}\grp{\op U}$ the opposite parabolic subgroup.

Fix a special maximal compact subgroup $K_0$ of $G$ (more precisely, the stabilizer of a special point in the apartment associated
to $\grp{T}_0$), so that we have the Iwasawa decomposition $P_0K_0=G$.
Also,
for any parabolic subgroup $\grp{P} = \grp{M} \grp{U}$ with Levi subgroup
$\grp{M} \supset \grp{M}_0$
we have $(P \cap K_0) = (M \cap K_0) (U \cap K_0)$.
Fix a faithful representation $\rho: \grp{G} \to \GL (V)$ (defined over $F$) and an $\OOO$-lattice $\Lambda_\rho$ in
the representation space $V$ such that $K_0 = \{ g \in G \, : \, \rho (g) \Lambda_\rho = \Lambda_\rho \}$, and
for $n=1,2,\dots$ let
\begin{equation} \label{EqnPrincipalCongruence}
K_n = K_{n,\rho} = \{ g \in G \, : \, \rho(g) v  \equiv v \pmod{\varpi^n \Lambda_\rho}, \quad v \in \Lambda_\rho \}
\end{equation}
be the associated principal congruence subgroups of $K_0$.

Suppose now that $\grp{P}=\grp{M}\grp{U}$ is a standard maximal parabolic subgroup.
Let $\chi_P$ be the fundamental weight of $\grp{P}$. Some integral power of $\chi_P$ defines a rational character of $\grp{P}$ trivial on $\grp{U}$.
Therefore $\abs{\chi_P}$ defines a character $\abs{\chi_P}:P\rightarrow\R_{>0}$ and we can extend it uniquely to a
right-$K_0$-invariant function, still denoted by $\abs{\chi_P}$, on $G$.
Let $\pi = (\pi,V_\pi)$ be an irreducible smooth representation of $M$ on
a complex vector space.
Let $\delta_P$ be the modulus function of $P$.
Consider the family of induced representations
$I_P(\pi,s)$, $s \in \C$, of $G$ which extend the $K_0$-representation $I^{K_0}_{P \cap K_0}(\pi|_{M \cap K_0})$.
Namely, $I_P (\pi,s)$ is the space of all smooth functions $\varphi: G \to V_\pi$ with
$\varphi (p g) = \abs{\chi_P} (p)^s \delta_P (p)^{1/2} \pi (p) \varphi (g)$ for all $p \in P$, $g\in G$, where $\pi$ is extended to $P$
via the canonical projection $P \to M$, and the $G$-action is given by right translations.
Any smooth function $\varphi: K_0 \to V_\pi$ with $\varphi (p k) = \pi (p) \varphi(k)$ for all $k \in P \cap K_0$ extends uniquely
to a function $\varphi_s \in I_P (\pi,s)$.
Let $\pi^\vee$ be the contragredient of $\pi$ and denote the pairing
between $V_\pi$ and $V_{\pi^\vee}$ by $(\cdot,\cdot)$. Then
\[
(\varphi,\varphi^\vee) = \int_{K_0} (\varphi (k), \varphi^\vee (k)) \, dk
\]
defines a pairing between $I_P (\pi,s)$ and $I_P (\pi^\vee,-s)$. Fix a Haar measure on $\op U$.
The intertwining operators $M(s)=M_{\op P|P}(\pi,s):I_P(\pi,s)\rightarrow I_{\op P}(\pi,-s)$,
which are defined by meromorphic continuation of the integrals
\[
(M (s) \varphi) (g) = \int_{\op U} \varphi ({\op u} g) \ d \op u, \quad \varphi \in I_P (\pi, s),
\]
were first studied in this generality by Harish-Chandra. (See \cite[Section IV]{MR1989693} for a self-contained treatment, cf. also
\cite{MR610479, MR544991}.)
It is known that the matrix coefficients $(M(s)\varphi_s,\varphi_{s}^\vee)$ for
$\varphi\in I^{K_0}_{P \cap K_0}(\pi|_{M \cap K_0})$ and $\varphi^\vee\in I^{K_0}_{{\op P} \cap K_0}(\pi^\vee|_{M \cap K_0})$ are rational functions
of $q^{-s}$ \cite[IV.1.1]{MR1989693} and that the degree of the denominator is bounded in terms of $\grp{G}$ only [ibid., IV.1.2].
We will recall this below (see Remark \ref{RemarkRationality}).

Fix an $\OOO$-lattice $\Lambda_\LieG \subset \LieG =\Lie \grp{G}$ stabilized by the operators $\Ad (k)$, $k \in K_0$.
Define a norm on $\LieG$ by $\norm{\sum_{i=1}^dt_iX_i}_{\LieG}=\max_{1 \le i \le d} \abs{t_i}$
for any $\OOO$-basis $X_1,\dots,X_d$ of $\Lambda_\LieG$. This defines a norm $\norm{\cdot}_{\End(\LieG)}$
on $\End(\LieG)$, namely $\norm{A}_{\End(\LieG)}$ is the maximum of the absolute values of the matrix
coefficients of $A$ with respect to the basis $X_1,\dots,X_d$. For any $g\in G$ we write $\norm{g}_G=\norm{\Ad(g)}_{\End(\LieG)}$
where $\Ad:\grp{G}\rightarrow\GL(\LieG)$
is the adjoint representation, and for any real number $R$ we set
\[
B(R)=\{g\in G:\norm{g}_G\le q^R\},
\]
which is a compact set modulo $Z$. We often omit the subscript from $\norm{\cdot}$ if it is clear from the context.

In the global situation of a reductive group $\grp{G}$ defined over a number field $k$,
we need of course to fix analogous global data that induce
the local data pertaining to $\grp{G} (k_v)$ for the non-archimedean places $v$ of $k$. In particular,
we fix a faithful representation $\rho: \grp{G} \to \GL (V)$ defined over $k$ and an $\OOO_k$-lattice $\Lambda_\rho$ in the $k$-vector space $V$,
and for every non-archimedean place $v$ of $k$ set $\Lambda_{\rho,v} = \Lambda_\rho \otimes_{\OOO_k} \OOO_{k_v} \subset V_v = V \otimes_{k}k_v$.
Using the base change $\rho_v$ of $\rho$ to $k_v$ and the lattice $\Lambda_{\rho,v}$,
we obtain open compact subgroups $K_{n,v} \subset \grp{G} (k_v)$, $n \ge 0$, as in \eqref{EqnPrincipalCongruence}.
It is well known that $K_{0,v}$ is then hyperspecial for almost all $v$ \cite[\S 3.9]{MR546588}, and by an appropriate choice of
$\Lambda_\rho$ we can ensure that it is special for all $v$.
We also fix an $\OOO_k$-lattice $\Lambda_\LieG \subset \LieG$ to define the local norms $\norm{\cdot}_{\grp{G} (k_v)}$ via base change to $\OOO_{k_v}$.

For a compact group $K$ let $e_K$ be the probability measure on $K$.
On any smooth $K$-representation on a complex vector space $e_K$ acts as the projector to the space of $K$-invariants.


\section{The main result}

We recall the definition of a fundamental boundedness property for the support of supercuspidal matrix coefficients \cite[Definition 7]{MR3001800}.
It is technically convenient to formulate it as follows. Recall that a smooth representation of the group $H$ of $F$-points of a reductive group $\grp{H}$ defined over $F$ is called quasicuspidal if its Jacquet modules with respect to all proper parabolic subgroups vanish.

\begin{definition} \label{DefinitionPSC}
A reductive group $\grp{H}$ defined over $F$ has property (PSC), if there exists a constant $c > 0$ such that for every
quasicuspidal representation $\pi$ of $H$ the support of the matrix coefficients $(\pi (h) v, v^\vee)$,
$v \in \pi^{K^H_n \cap H^{\der}}$, $v^\vee \in (\pi^\vee)^{K^H_n \cap H^{\der}}$, is contained in the set
$B^H (c n)$ for any $n \ge 1$.
\end{definition}

Note that property (PSC) for a group $\grp{H}$ is equivalent to property (PSC) for its derived group $\grp{H}^{\der}$. (It is really a property of semisimple groups. The extension to general reductive groups is only for technical reasons.)
Moreover, it suffices to consider only irreducible supercuspidal representations $\pi$ in Definition \ref{DefinitionPSC} (as was done in \cite{MR3001800}). We also note that property (PSC)
does not depend on the choice of the representation $\rho$ used to define $K^H_n$ and the norm $\norm{\cdot}_{\mathfrak{h}}$
on the Lie algebra of $\grp{H}$, although the possible values of $c$ will depend on these choices.
In the following, we will consider property (PSC) only for semisimple
subgroups $\grp{H}$ of $\grp{G}$, and use $K^H_n = K_n \cap H$ and the restriction of the fixed norm on $\LieG$.

We have the following simple compatibility results for property (PSC).

\begin{lemma} \label{LemmaPSCComp}
\begin{enumerate}
\item If a group $H$ satisfies property (PSC), then any quotient by a central subgroup also satisfies property (PSC).
\item Any direct product of groups satisfying property (PSC) also satisfies property (PSC).
\end{enumerate}
\end{lemma}

The main result of this paper is the following.

\begin{theorem} \label{TheoremBD}
Asssume that the connected semisimple normal subgroups of all proper Levi subgroups $M$ of $G$ satisfy property (PSC).
Then the group $G$ satisfies property (BD).
\end{theorem}

By Lemma \ref{LemmaPSCComp}, it is enough to assume (PSC) for the isotropic connected $F$-simple normal subgroups of proper Levi subgroups $\grp{M}$.
They correspond to the non-trivial connected proper subdiagrams of the Dynkin diagram of $\grp{G}$ over $F$.

By a standard procedure, we can reduce to the following statement.
Recall that we denote by $\grp{\hat M}$ the subgroup of $\grp{G}$ generated by $\grp{U}$ and $\grp{\op U}$. It is an $F$-simple normal subgroup of $\grp{G}$.

\begin{theorem} \label{TheoremSupercuspidalBD}
Let $\grp{P}$ be a maximal parabolic subgroup of $\grp{G}$ and assume that the connected semisimple normal subgroups of $\grp{M}$ satisfy property (PSC).
Then there exists a constant $C>0$ such that for every irreducible supercuspidal representation $\pi$ of $M$ we have
\[
\deg M_{\bar{P}|P} (\pi, s)^{K_n \cap \hat{M}} \le C n \quad \text{for all}\quad n \ge 1.
\]
\end{theorem}

Before we proceed further, we derive Theorem \ref{TheoremBD} from Theorem
\ref{TheoremSupercuspidalBD}.

\begin{proof}[Proof of Theorem \ref{TheoremBD}]
Upon replacing $G$ by a Levi subgroup, it is clearly sufficient to show property (BDmax).
For this, just copy the proof of
\cite[Lemma 20]{MR3001800}
for the closed normal subgroup $K = K_n \cap \hat{M}$ of $K_0$ (it is irrelevant that this is in general not an open subgroup of $K_0$).
We reduce to the degree bound
\[
\deg M_{\bar{Q'}|Q'} (\sigma, s)^{K_n \cap \hat{M} \cap M_R} \le C n,
\quad n \ge 1,
\]
where $\sigma$ is an irreducible supercuspidal representation of a Levi subgroup $L$ of $G$, which has co-rank one in another Levi subgroup $M_R$,
$Q'$ is a (maximal) parabolic subgroup of $M_R$ with Levi subgroup $L$,
and $M_{\bar{Q'}|Q'} (\sigma, s)$ the associated intertwining operator between representations of $M_R$.
Moreover, we can assume that $\sprod{\chi_P}{\alpha^\vee} \neq 0$, if $\alpha \in \Sigma_L$ defines the Levi subgroup $M_R$.
The latter condition implies that $\hat{M}_R \subset \hat{M} \cap M_R$, and the required bound follows therefore from Theorem \ref{TheoremSupercuspidalBD}.
\end{proof}

We now turn to the proof of Theorem \ref{TheoremSupercuspidalBD}, which
will occupy the remainder of this section. The strategy of the proof is based on the standard analysis of the restriction of an induced representation to
the opposite parabolic subgroup $\bar{P}$ going back to Bernstein-Zelevinsky (cf. \cite[\S I.3]{MR1989693}).
All constants appearing until the end of this section may depend on $\grp{G}$, $\grp{T}_0$, $\rho$, $\Lambda_\rho$, and the norm
$\norm{\cdot}_{\mathfrak{g}}$ on $\LieG$, but are supposed to be independent
of all other data, in particular the supercuspidal representation $\pi$.

Consider the $(\grp{P},\grp{\op P})$-Bruhat decomposition of $\grp{G}$.
The double classes are parametrized by elements $w_1, \ldots, w_k \in N_G (T_0) \cap K_0$.
(We can take the representatives in $K_0$, because $K_0$ is supposed to be special.)
We fix an ordering of the double cosets such that
the Zariski closure of each coset $\grp{P} w_i \grp{\op P}$ is contained in the union of the cosets $\grp{P} w_j \grp{\op P}$ for $j \ge i$.
This means that the sets $\grp{Z}_i = \bigcup_{j > i} \grp{P} w_j \grp{\op P}$,
$i = 0, \ldots, k$, are Zariski closed.
They form a descending chain
\begin{equation} \label{EquationZi}
\grp{Z}_0 = \grp{G} \supset \grp{Z}_1 \supset \ldots \supset \grp{Z}_{k-1} = \grp{P} w_k \grp{\op P} \supset
\grp{Z}_k = \emptyset.
\end{equation}
In particular, $\grp{P} w_1 \grp{\op P} = \grp{P} \grp{\op P}$ is the big cell, $\grp{Z}_1$ is its complement in $\grp{G}$,
and $\grp{P} w_k \grp{\op P}$ is the unique closed double coset. The latter is a single coset precisely if $\grp{P}$ is conjugate to $\grp{\op P}$.

Before we begin with the main part of the argument, we collect a few simple facts that we will need.

\begin{itemize}
\item The intersection $w_i^{-1} \grp{P} w_i \cap \grp{\op P}$ is the semidirect product of
$w_i^{-1} \grp{P} w_i \cap \grp{M}$ and $w_i^{-1} \grp{P} w_i \cap \grp{\op U}$, and
$w_i^{-1} \grp{P} w_i \cap \grp{M}$ is a parabolic subgroup of $\grp{M}$ containing $\grp{T_0}$.
Also, $w_i \grp{\op P} w_i^{-1} \cap \grp{M}$ is a parabolic subgroup of $\grp{M}$ with unipotent radical
$\grp{V} = w_i \grp{\op U} w_i^{-1} \cap \grp{M}$ and Levi subgroup $\grp{L} = w_i \grp{M} w_i^{-1} \cap \grp{M}$.
Note that $\grp{L} = \grp{M}$ (or equivalently, $\grp{V}$ is trivial) precisely if either $i=1$ or $i = k$ and $\grp{\op P}$ is conjugate to $\grp{P}$.

\item Let $\tau_P \in T_M$ be a generator of $T_M$ modulo $Z(G) (Z(M) \cap K_0)$ such that $\abs{\chi_P} (\tau_P)< 1$.
We have $\abs{\chi_P} (\tau_P) = q^{-m_P}$ with a positive rational number $m_P$, which can be bounded, together with its denominator, solely
in terms of the root system $\rts = R(\grp{T}_0,\grp{G})$
of $\grp{G}$ over $F$. There exists a positive integer $C_1$ such that
\begin{equation} \label{EquationConjugationTau}
\tau_P K_{n+C_1} \tau_P^{-1} \subset K_n,  \quad \tau_P^{-1} K_{n+C_1} \tau_P \subset K_n \quad \text{for all} \quad n \ge 0.
\end{equation}
(Namely, if $\rho (\tau_P)$ and $\rho (\tau_P)^{-1}$ map the lattice $\Lambda_\rho$ into $\varpi^{-k} \Lambda_\rho$
for some positive integer $k$, then we can take $C_1=2k$.)

\item Set
\[
\bar{U} (m) = \tau_P^{m} (\bar{U} \cap K_0) \tau_P^{-m}, \quad m \ge 0.
\]
The subgroups $\bar{U} (m)$ of $\bar{U}$ are normalized by $K_0  \cap M$ and form an exhausting filtration of $\bar{U}$. On the other hand, we have
\begin{equation}
\label{EquationConjugationU}
\tau_P^{m} (U \cap K_0) \tau_P^{-m} \subset U \cap K_{m-C_2+1}, \quad m \ge C_2-1,
\end{equation}
for a suitable positive integer $C_2$. (This follows easily by identifying
$U$ with its Lie algebra via the exponential map.)

\item Write $\abs{\chi_P} (w_k \tau_P^{-1} w_k^{-1}) = q^{-m'_P}$, where
$m'_P = m_{w_k \bar{P} w_k^{-1}} > 0$ is again bounded, together with its denominator, solely in terms of $\rts$.
(Note that $w_k \grp{\op P} w_k^{-1}$ is a standard parabolic subgroup.)
As a consequence of (the $p$-adic analog of) Kostant convexity \cite[Proposition 4.4.4]{MR0327923}, we have then
\begin{equation} \label{EquationChiPUm}
1 \ge \abs{\chi_P} (\bar{u}) \ge q^{- (m_P+m'_P) m} \quad \text{for all} \quad\bar{u} \in \bar{U} (m).
\end{equation}
\item We also note the following elementary facts:
\begin{eqnarray}
\bar{U} \cap w_i^{-1} P w_i  M
& = & \bar{U} \cap w_i^{-1} P w_i, \label{EquationBarUCap}\\
w_i \bar{U} (m) w_i^{-1} \cap P & = & (w_i \bar{U} (m) w_i^{-1} \cap M)
(w_i \bar{U} (m) w_i^{-1} \cap U) \label{EquationBarUmCap}
\end{eqnarray}
for all $1 \le i \le k$ and $m \ge 0$.
\end{itemize}

The basic fact underlying our argument is the following geometric lemma.

\begin{lemma} \label{LemmaGeometric}
For all $1 \le i \le k$ and all $n \ge 1$ we have
\[
P w_i \bar{P} \subset Z_i (U \cap K_n) \cup P w_i \bar{U} (C_2 n) (M \cap K_0).
\]
\end{lemma}

\begin{proof}
We start with a consequence of Bruhat-Tits theory \cite[\S 4]{MR0327923}.
Recall that $K_0$ is the stabilizer of a special point $x$ in the apartment of $\grp{T}_0$. There exist a chamber $C$ of this apartment and a facet $y$ of $C$, both containing $x$, such that the stabilizer $G_y$ of $y$ and the
stabilizer $\bar{I} = G_C$ of $C$ are related by
$G_y = (\bar{P} \cap K_0) \bar{I} = \bar{I} (\bar{P} \cap K_0)$.
We have then
\[
G = \bigcup_{i=1}^k
P w_i (\bar{P} \cap K_0) \bar{I} = \bigcup_{i=1}^k
P w_i (\bar{P} \cap K_0) (U \cap \bar{I}).
\]

Moreover, if $P w_i \bar{P}$ intersects
$P w_j (\bar{P} \cap K_0) \bar{I}$, then $\grp{P} w_j \grp{\op P}$ lies in the Zariski closure of $\grp{P} w_i \grp{\op P}$,
and in particular $j \ge i$.
To see this, write $(\bar{P} \cap K_0) \bar{I} = (U \cap \bar{I}) (\bar{P} \cap K_0)$, and observe that a non-empty intersection is equivalent to
the existence of $u \in U \cap \bar{I}$, $p \in P$ and $\bar{p}\in \bar{P}$ with $w_j u = p w_i \bar{p}$. But then we have also
$w_j (\tau_P^m u \tau_P^{-m}) = (w_j \tau_P^m w_j^{-1} p) w_i (\bar{p} \tau_P^{-m}) \in P w_i \bar{P}$ for any integer $m$.
Letting $m \to \infty$, we obtain that $w_j$ lies in the topological closure of the Bruhat cell $P w_i \bar{P}$, as claimed.

Finally,
\[
P w_i \bar{P} \cap
P w_i (\bar{P} \cap K_0) \bar{I} =
P w_i (\bar{P} \cap K_0).
\]
For this, we reduce again to the case where $w_i u \in P w_i \bar{P}$ for $u \in U  \cap \bar{I}$.
We can write $u = u_1 u_2$ with $u_1 \in U  \cap w_i^{-1} P w_i \cap \bar{I}$
and $u_2 \in U  \cap w_i^{-1} \bar{U} w_i \cap \bar{I}$. We obtain $w_i u_2 \in P w_i \bar{P}$, which implies that $u_2 = e$.
Therefore $w_i u = w_i u_1 = (w_i u_1 w_i^{-1}) w_i \in P w_i$.

Combining the previous results, we obtain
\[
P w_i \bar{P} \subset
\bigcup_{j > i}
P w_j (\bar{P} \cap K_0) (U \cap \bar{I})
\cup
P w_i (\bar{P} \cap K_0)
\subset Z_i (U \cap \bar{I}) \cup
P w_i (\bar{P} \cap K_0).
\]
Multiplying by $\tau_P^{-m}$ from the right, and observing that $Z_i$ and $P w_i$ are right $T_0$-invariant, we get
\[
P w_i \bar{P}
\subset Z_i \tau_P^{m} (U \cap \bar{I}) \tau_P^{-m} \cup
P w_i \bar{U} (m) (M \cap K_0),
\]
which, after invoking \eqref{EquationConjugationU}, easily yields the assertion of the lemma (even with $n+C_2-1$ instead of $C_2 n$).
\end{proof}

We also need the following reformulation of property (PSC) (which is of course standard at least in a qualitative form).

\begin{lemma} \label{LemmaPSC}
There exists a constant $C_3 \ge 0$ with the following property.
Let $\grp{Q} = \grp{L} \grp{V}$ be a proper parabolic subgroup of $\grp{M}$ containing $\grp{T}_0$, and assume that the semisimple normal subgroup $\grp{\tilde M}$
of $\grp{M}$ generated by $\grp{V}$ and $\grp{\op V}$ satisfies property (PSC) with a constant $c \ge 1$. Then for every integer $n \ge 0$,
every quasicuspidal representation $\pi$ of $M$, and every $b \in Z (L)$ with $|\alpha| (b) \le q^{-C_3}$
for all $\alpha \in \Phi_V$ and
$\min_{\alpha \in \Phi_V} |\alpha| (b) < q^{-c(n+C_3)}$
we have the inclusion
\[
\pi^{K_n \cap \tilde M} \subset \ker \pi (e_{b^{-1} (V \cap K_n) b}).
\]
\end{lemma}

\begin{proof}
Assume first that $b$ is contained in $\tilde M$.
We take $C_3 \ge 0$ such that for every $n \ge 0$ there exists an open compact subgroup $\tilde K_n$ of $\tilde M$ with
\[
K_{n+C_3} \cap \tilde M \subset \tilde K_n
\subset K_{n} \cap \tilde M,
\]
that is totally decomposed, i.e. satisfies
\[
\tilde K_n = (\tilde K_n \cap \tilde M \cap M_0) \prod_{\alpha \in \Sigma_{\tilde M}} (\tilde K_n \cap U_{\alpha})
\]
(in any order).
To show the existence of $\tilde K_n$,
let $e$ be the absolute ramification degree of $F$ (i.e., $p \varpi^{-e}$ is a unit in $\OOO$)
and note that if we identify $\grp{\tilde M}$ with its image in $\GL (V_\rho)$ via $\rho$, then
for $n > e / (p-1)$ the $p$-adic logarithm map maps $K_{n} \cap \tilde M$ bijectively onto $\varpi^n \Lambda_{\tilde M}$ for a certain $\OOO$-Lie lattice
$\Lambda_{\tilde M}$ in the Lie algebra of $\grp{\tilde M}$, regarded as a subspace of the $F$-vector space $\mathfrak{gl} (V_\rho)$.
It remains to take an open $\OOO$-Lie sublattice $\Lambda^{\rm td}_{\tilde M} \subset \Lambda_{\tilde M}$ satisfying
\[
\Lambda^{\rm td}_{\tilde M} = \Lambda^{\rm td}_{\tilde M} \cap
\Lie (\grp{\tilde M} \cap \grp{M_0}) + \sum_{\alpha \in \Phi_{\tilde M}}
(\Lambda^{\rm td}_{\tilde M} \cap \mathfrak{u}_\alpha),
\]
and to let $\tilde K_n = \exp (\varpi^n \Lambda^{\rm td}_{\tilde M})$ for $n > e / (p-1)$ and
$\tilde K_n = \exp (\varpi^{\lceil  e / (p-1) \rceil} \Lambda^{\rm td}_{\tilde M})$, otherwise.

The decomposition of $\tilde K_n$ implies that
\[
\tilde K_n = (\tilde K_n \cap V) (\tilde K_n \cap L \cap \tilde M)
(\tilde K_n \cap \bar{V})
\]
and
\[
\tilde K_n \cap \bar{V} = \prod_{\alpha \in \Sigma_V} (\tilde K_n \cap \bar{V}_{-\alpha}).
\]
Let $v$ be a $K_n \cap \tilde M$-invariant vector in the space of $\pi$.
By our assumption on property (PSC) for $\tilde M$, we have
$\pi (e_{K_m \cap \tilde M}) \pi (b) v = 0$ and therefore also
$\pi (e_{b^{-1} (K_m \cap \tilde M) b}) v = \pi (b)^{-1} \pi (e_{K_m \cap \tilde M}) \pi (b) v = 0$ for $m \ge n$
if $\norm{b} > q^{cm}$.
We apply this for $m = n+C_3$ and deduce that
$\pi (e_{b^{-1} \tilde K_n b}) v = 0$, since by assumption $\norm{b}
\ge \max_{\alpha \in \Phi_V} |\alpha| (b)^{-1}
> q^{c(n+C_3)}$.
From the factorization of $\tilde K_n$ we get that
\[
e_{b^{-1} \tilde K_n b} = e_{b^{-1} (\tilde K_n \cap V) b} e_{\tilde K_n \cap L \cap \tilde M} e_{b^{-1}
(\tilde K_n \cap \bar{V}) b}.
\]
By our assumption on $b$, we have
$b^{-1} (\tilde K_n \cap \bar{V}_{-\alpha}) b \subset \tilde K_n \cap \bar{V}_{-\alpha}$ for all $\alpha \in \Sigma_V$ and therefore
$b^{-1} (\tilde K_n \cap \bar{V}) b \subset \tilde K_n \cap \bar{V}$. Since $v$ is $\tilde K_n$-invariant, we conclude that
$\pi (e_{b^{-1} (K_n \cap V) b}) v = 0$, as required.

It remains to consider the case where $b$ is not necessarily contained in
$\tilde M$.
For this it suffices to observe that with a suitable choice of $C \ge 0$ for every $b \in Z(L)$ there exists an element $b' \in Z(L) \cap \tilde M$ with
$\abs{\alpha} (b) \le \abs{\alpha} (b') \le \abs{\alpha} (b) q^C$ for all
$\alpha \in \Phi_V$. In particular, this implies that $(b')^{-1} (K_n \cap V) b'$ is contained
in $b^{-1} (K_n \cap V) b$. Applying the previous argument to $b'$, with $C_3$ replaced by $C_3+C$, yields the assertion.
\end{proof}

We now consider the ascending filtration of the space $I_P (\pi,s)$ by $\bar{P}$-invariant subspaces corresponding to the descending chain \eqref{EquationZi}:
\[
I_P (\pi,s)_i = \{ \varphi \in I_P (\pi,s) \, : \, \varphi|_{Z_i} = 0 \}, \quad i = 0, \ldots, k.
\]
The first non-trivial space $I_P (\pi,s)_1$ is the space of all sections with support contained in the big cell $P \bar{P}$.

Consider the projector $e_{\bar{U} (m)}$ acting on $I_P (\pi,s)$. Clearly, $e_{\bar{U} (m)}$ maps each space $I_P (\pi,s)_i$ to itself.
We note the following simple consequence of \eqref{EquationConjugationTau}:
\begin{equation} \label{EquationLevel}
\text{The idempotent $e_{\bar{U} (m)}$ maps $K_n \cap \hat{M}$-invariants to $K_{n+2 C_1 m} \cap \hat{M}$-invariants.}
\end{equation}

\begin{lemma} \label{LemmaInductionStep}
There exists a non-negative constant $C_4$ (in fact, we can take $C_4 = 2+2C_1 C_2+C_3$) with the following property.
Assume that the connected semisimple normal subgroups of $M$ satisfy property (PSC) with a constant $c_M \ge 1$.
Let $2 \le i \le k$, with $i < k$ in case $\grp{P}$ is conjugate to $\grp{\op P}$.
Then for every $K_n \cap \hat{M}$-invariant function $\varphi \in I_P (\pi,s)_i$, the function
$e_{\bar{U} (m)} \varphi$ belongs to the space
$\varphi \in I_P (\pi,s)_{i-1}$ for all
$m \ge C_4 c_M n$.
\end{lemma}

\begin{proof}
Let $\varphi \in I_P (\pi,s)_i$ be invariant under $K_n \cap \hat{M}$. Since $e_{\bar{U} (m)}$ acts on $I_P (\pi,s)_i$ for any $m$,
and $Z_{i-1}$ is the union of $Z_i$ and the double coset $P w_i \bar{P}$, we only have to show that
$e_{\bar{U} (m)} \varphi$ vanishes on $P w_i \bar{P}$, or equivalently, that
\begin{equation} \label{EquationeUm}
\int_{\bar{U} (m)} \varphi (w_i \bar{p} \bar{u}) \, d \bar{u} = 0,
\quad \bar{p} \in \bar{P}, \, m \ge C_4 c_M n.
\end{equation}

By the definition of the space $I_P (\pi,s)_i$, the function $\varphi$ vanishes on the set $Z_i (U \cap K_n)$.
Applying Lemma \ref{LemmaGeometric}, we conclude that
\begin{equation} \label{EquationSupport}
\supp \varphi \cap P w_i \bar{P} \subset P w_i \bar{U} (C_2 n) (K_0 \cap M).
\end{equation}
We take $C_4 \ge C_2$, which ensures that $m \ge C_4 c_M n$ implies $m \ge C_2 n$.
Under this condition, \eqref{EquationeUm} reduces to
\[
\int_{\bar{U} (m)} \varphi (w_i \bar{u} k) \, d \bar{u} = 0, \quad k \in K_0 \cap M, \, m \ge C_4 c_M n.
\]
Combining \eqref{EquationSupport} with \eqref{EquationBarUCap}, we are reduced to showing that
\[
\int_{(\bar{U} (m) \cap w_i^{-1} P w_i) \bar{U} (C_2 n)} \varphi (w_i  \bar{u} k) \, d \bar{u} = 0, \quad k \in K_0 \cap M, \, m \ge C_4 c_M n,
\]
which will clearly follow from the stronger statement that
\[
\int_{\bar{U} (m) \cap w_i^{-1} P w_i} \varphi (w_i \bar{u} \bar{u}_0 k) \, d \bar{u} = 0, \quad k \in K_0 \cap M, \bar{u}_0 \in \bar{U} (C_2 n),
\, m \ge C_4 c_M n.
\]
Conjugating $\bar{u}$ by $w_i$, the integral here is equal to
\[
\int_{w_i \bar{U} (m) w_i^{-1} \cap P} \varphi (\bar{u} w_i \bar{u}_0 k) \, d \bar{u},
\]
which because of \eqref{EquationBarUmCap} is equal to a constant multiple of $e_{w_i \bar{U} (m) w_i^{-1} \cap M}(\varphi (w_i \bar{u}_0 k))$.

It follows easily from \eqref{EquationConjugationTau}, that conjugation by $(w_i \bar{u}_0 k)^{-1}$ maps
$K_{(1+2 C_1 C_2)n} \cap M \cap \hat{M}$ into $K_n \cap \hat{M}$. Therefore,
the element $\varphi (w_i \bar{u}_0 k)$ of the space of $\pi$ is invariant under $K_{(1+2 C_1 C_2)n} \cap M \cap \hat{M}$.
We claim that $e_{w_i \bar{U} (m) w_i^{-1} \cap M}$ annihilates the elements of
$\pi^{K_{(1+2 C_1 C_2)n} \cap M \cap \hat{M}}$ for all $m \ge C_4 c_M n$, which will finish the argument.
To see this, apply Lemma \ref{LemmaPSC} to the parabolic subgroup
$\grp{Q} = w_i \grp{\op P} w_i^{-1} \cap \grp{M}$ of $\grp{M}$
with unipotent radical $\grp{V} = w_i \grp{\op U} w_i^{-1} \cap \grp{M}$
and to the element $b = w_i \tau_P^{-m} w_i^{-1}$.
Our restriction on $i$ implies that $\grp{Q}$ is a proper parabolic subgroup of $\grp{M}$.
Note also that the normal subgroup $\grp{\tilde M}$ of Lemma \ref{LemmaPSC} is contained in the intersection $\grp{M} \cap \grp{\hat{M}}$.
We have $\abs{\alpha (b)} = \abs{\alpha (w_i \tau_P^{-m} w_i^{-1})}
= \abs{\alpha (w_i \tau_P^{-1} w_i^{-1})}^m \le q^{-m}$ for all $\alpha \in \Phi_V$.
Taking $C_4 = 2+2C_1 C_2+C_3$ ensures that
$\abs{\alpha (b)} < q^{-c_M ((1+2 C_1 C_2)n + C_3)}$ for all
$\alpha \in \Phi_V$. With this choice of $C_4$ the hypotheses of Lemma \ref{LemmaPSC} are satisfied, which establishes our claim and finishes the proof.
\end{proof}

\begin{lemma} \label{LemmaSupport}
Assume that the connected semisimple normal subgroups of $M$ satisfy property (PSC) with a constant $c_M \ge 1$.
Then there exists a positive constant $C_5$ (in fact, we may take $C_5 = k C_2 (2 C_1 C_4 c_M)^{k-1}$) with the following property.
For every $K_n \cap \hat{M}$-invariant function $\varphi \in I_P (\pi,s)$, which in addition satisfies $\varphi \in I_P (\pi,s)_{k-1}$
in case $\grp{P}$ is conjugate to $\grp{\op P}$, the function $e_{\bar{U} (m)} \varphi$ is supported in the set $P \bar{U} (m)$
for all $m \ge C_5 n$.
\end{lemma}

\begin{proof} By \eqref{EquationLevel}, the function $e_{\bar{U} (m)} \varphi$ is
invariant under $K_{n+2 C_1 m} \cap \hat{M}$ for any $m \ge 0$.
By downward induction on $i = k-1, \ldots, 1$, we can now derive from Lemma \ref{LemmaInductionStep} the following statement:
\[
e_{\bar{U} (m)} \varphi  \in I_P (\pi,s)_{i}^{K_{n+2 C_1 m} \cap \hat{M}},
\quad m \ge m_i = C_4 c_M n \sum_{j=0}^{k-1-i} (2 C_1 C_4 c_M)^j.
\]
Indeed, the case $i=k-1$ is trivial if $\grp{P}$ and $\grp{\op P}$ are conjugate, and follows directly from Lemma \ref{LemmaInductionStep}
otherwise, while for the induction step we apply the lemma to
$e_{\bar{U} (m_i)} \varphi$ and $m \ge m_{i-1} \ge m_i$, and observe that
$e_{\bar{U} (m)} \varphi = e_{\bar{U} (m)} e_{\bar{U} (m_i)} \varphi$ for $m \ge m_i$.

The end result for $i = 1$ is that the function $e_{\bar{U} (m_1)} \varphi$ is $K_{n+2 C_1 m_1} \cap \hat{M}$-invariant and supported on the
big cell $P \bar{P}$. By Lemma \ref{LemmaGeometric}, we have
\[
P \bar{P} \subset (G - P \bar{P}) (K_{n+2 C_1 m_1} \cap \hat{M}) \cup P \bar{U} (C_2 (n+2 C_1 m_1)).
\]
Therefore, $e_{\bar{U} (m_1)} \varphi$ is actually supported on $P \bar{U} (m_0)$, where
\[
m_0 = C_2 (n+2 C_1 m_1) = C_2 n \sum_{j=0}^{k-1} (2 C_1 C_4 c_M)^j.
\]
We obtain the assertion with
$C_5 = C_2 \sum_{j=0}^{k-1} (2 C_1 C_4 c_M)^j \le k C_2 (2 C_1 C_4 c_M)^{k-1}$.
\end{proof}

We can now prove Theorem \ref{TheoremSupercuspidalBD}. We follow
the proof of \cite[Theorem 21]{MR3001800}.

\begin{proof}[Proof of Theorem \ref{TheoremSupercuspidalBD}]
Let $\pi$ be a supercuspidal representation of $M$ and
$\varphi$ be an element of the space $I^{K_0}_{P \cap K_0} (\pi|_{M \cap K_0})^{K_n \cap \hat{M}}$.
The function $\varphi$ has a unique extension to a function $\varphi_s \in I_P (\pi,s)^{K_n \cap \hat{M}}$.
By definition, we have
\begin{equation} \label{EquationMpis}
(M(\pi,s) \varphi_s) (k) = \int_{\bar U} \varphi_s (\bar{u} k) \, d \bar{u},
\quad k \in K_0.
\end{equation}
Assume first that $\grp{P}$ and $\grp{\op P}$ are not conjugate.
In this case, we can replace the integration over $\bar{U}$ in \eqref{EquationMpis} by integration over the compact group $\bar{U} (C_5 n)$:
\begin{equation} \label{EquationMpisRestricted}
(M(\pi,s) \varphi_s) (k) = \int_{\bar U (C_5 n)} \varphi_s (\bar{u} k) \, d \bar{u},
\quad k \in K_0.
\end{equation}
To see this, observe that \eqref{EquationMpis} clearly implies that
\begin{equation} \label{EquationEU}
(M(\pi,s) \varphi_s) (k) = \int_{\bar U} (e_{\bar{U} (C_5 n)} I (k) \varphi_s) (\bar{u}) \, d \bar{u}.
\end{equation}
Applying Lemma \ref{LemmaSupport}
to the $K_n \cap \hat{M}$-invariant function $I (k) \varphi_s = \varphi_s (\cdot k)$, $k \in K_0$, shows that the support of the integrand in
\eqref{EquationEU} is contained in $P \bar{U} (C_5 n) \cap \bar{U} = \bar{U} (C_5 n)$, which establishes \eqref{EquationMpisRestricted}.

Let now
$\varphi^\vee \in I^{K_0}_{\op P \cap K_0} (\pi^\vee|_{M \cap K_0})$,
and extend this function to $\varphi^\vee_{s} \in I_{\op P} (\pi^\vee, s)$. Using \eqref{EquationMpisRestricted},
the matrix coefficient $(M(\pi,s) \varphi_s, \varphi_{s}^\vee)$ can be computed as
\[
(M(\pi,s) \varphi_s, \varphi_{s}^\vee)
= \int_{K_0} ((M(\pi,s) \varphi_s) (k), \varphi^\vee (k)) \ dk
= \int_{\op U (C_5 n)} \abs{\chi_P} (\op u)^s f (\op u) \ d\op u
\]
with
\[
f (\op u) = \int_{K_0} (\varphi_0 (\op u k), \varphi^\vee (k)) \ dk.
\]
Using \eqref{EquationChiPUm}, we conclude that the matrix coefficient
$(M(\pi,s) \varphi_s, \varphi_{s}^\vee)$ is a polynomial in $q^{-s}$ of degree $\le (m_P+m'_P) C_5 n$.
This finishes the case where $\grp{P}$ and $\grp{\op P}$ are not conjugate.

We now consider the case where $\grp{P}$ and $\grp{\op P}$ are conjugate.
We first remark that as a consequence of Lemma \ref{LemmaSupport} we still have
\begin{equation} \label{EquationMpisRestricted2}
(M(\pi,s) \psi_s) (e) = \int_{\bar U (C_5 n)} \psi_s (\bar{u}) \, d \bar{u}
\quad \text{for all} \quad \psi_s \in I (\pi, s)_{k-1}^{K_n \cap \hat{M}}.
\end{equation}
Moreover, the filtration space $I (\pi, s)_{k-1}$ is simply given by
\[
I (\pi, s)_{k-1} = \{ \psi_s \in I (\pi,s) \, : \, \psi_s (w_k) = 0 \}.
\]

Let $\varphi$ and $\varphi^\vee$ be as above. We will choose
$a \in Z(M)$ and $b = w_k^{-1} a w_k$ as follows (recall that $w_k$ normalizes $M$): if
$\omega_\pi |_{Z(M) \cap K_0} \neq \omega_\pi \circ w_k^{-1} |_{Z(M) \cap K_0}$, where $w_k^{-1}$ acts on $Z(M)$ by conjugation,
then take $a \in Z(M) \cap K_0$ with $\omega_\pi (a) \neq \omega_\pi (b)$.
Otherwise set $a = \tau_P$ and recall that
$|\chi_P| (a) = q^{-m_P}$. Under our assumption that $\grp{P}$ and $\grp{\op P}$ are conjugate, $2 m_P$ is actually a positive integer.
Consider the difference operator
\[
\Delta_{a,s} = \omega_s (b)^{-1} (\delta_P (a)^{-1/2} I (b,s) - \omega_s (a) \Id)
\]
acting on $I_P (\pi,s)$,
where $\omega_s = \omega_\pi |\chi_P|^s$.
It has the following two crucial properties (cf. \cite[p. 448]{MR3001800}):
the image of $\Delta_{a,s}$ is contained in the space $ I (\pi, s)_{k-1}$, and
\[
(M (\pi, s) \Delta_{a,s} \varphi_s) (e) = (1 - \omega_s (b^{-1} a))
(M (\pi, s) \varphi_s) (e),
\]
where $1 - \omega_s (b^{-1} a)$ does not vanish identically for all $s$ by our choice of $a$.
Applying this relation to $I (k)  \varphi_s$, we obtain
\[
(M (\pi, s) \Delta_{a,s} I (k) \varphi_s) (e) = (1 - \omega_s (b^{-1} a))
(M (\pi, s) \varphi_s) (k), \quad k \in K_0.
\]
If we set
\[
\psi_{s,k} = \Delta_{a,s} I (k) \varphi_s \in I (\pi, s)_{k-1},
\]
then $\psi_{s,k}$ is $K_{n'} \cap \hat{M}$-invariant, where $n'=n$ for
$a \in K_0$ and $n'=n+C_1$ for $a = \tau_P$ (using \eqref{EquationConjugationTau}). Therefore we get from
\eqref{EquationMpisRestricted2} that
\[
(M(\pi,s) \psi_{s,k}) (e) = \int_{\bar U (C_5 n')} \psi_{s,k} (\bar{u}) \, d \bar{u}.
\]

We can now compute the matrix coefficent $(M(\pi,s) \varphi_s, \varphi_{s}^\vee)$ as follows:
\[
(1 - \omega_s (b^{-1} a)) (M(\pi,s) \varphi_s, \varphi_{s}^\vee)
= \int_{K_0} ((M(\pi,s) \psi_{s,k}) (e), \varphi^\vee (k)) \ dk
= \int_{\op U (C_5 n')} \abs{\chi_P} (\op u)^s f (\op u) \ d\op u
\]
with
\[
f (\op u) = \int_{K_0} (\psi_{0,k} (\op u), \varphi^\vee (k)) \ dk.
\]
In the case
$\omega_\pi |_{Z(M) \cap K_0} \neq \omega_\pi \circ w_k^{-1} |_{Z(M) \cap K_0}$,
we can therefore conclude as above that the matrix coefficient
$(M(\pi,s) \varphi_s, \varphi_{s}^\vee)$ is a polynomial in $q^{-s}$ of degree $\le (m_P+m'_P) C_5 n$.

In the remaining case, we obtain that the product
\[
(1 - \omega_\pi (b^{-1} a) (q^{-s})^{2m_P}) (M(\pi,s) \varphi_s, \varphi_{s}^\vee)
\]
is a polynomial in $q^{-s}$ of degree $\le (m_P+m'_P) C_5 n' = (m_P+m'_P) C_5 (n+C_1) \le (m_P+m'_P) C_5 (C_1 + 1) n$.
This finishes the proof.
\end{proof}

\begin{remark} \label{RemarkConstant}
The proof shows that the constant $C$ in Theorem \ref{TheoremSupercuspidalBD} can be taken to be
$(m_P+m'_P) C_5 (C_1 +1)$, if $\grp{P}$ and $\grp{\op P}$ are conjugate, and $(m_P+m'_P) C_5$, otherwise.
\end{remark}

\begin{remark} \label{RemarkRationality}
Without any assumption, the same proof establishes that the matrix coefficients of $M(\pi,s)$ are polynomials in $q^{-s}$ in the cases where
either $\grp{P}$ and $\grp{\op P}$ are not conjugate or they are conjugate but
$\omega_\pi |_{Z(M) \cap K_0} \neq \omega_\pi \circ w_k^{-1} |_{Z(M) \cap K_0}$, and that the matrix coefficients of
$(1 - \omega_\pi (w_k^{-1} \tau_P^{-1} w_k \tau_P) (q^{-s})^{2m_P}) M(\pi,s)$ are polynomials in $q^{-s}$ in the remaining case.
\end{remark}

\section{Global uniformity and limit multiplicities}
We recall that in \cite[Corollary 13]{MR3001800} the following result on property (PSC) was obtained
(the second part is based on Kim's result \cite{MR2276772} on the exhaustiveness of Yu's construction of supercuspidal representations
\cite{MR1824988} for large residual characteristic).

\begin{theorem} \label{TheoremPSC}
\begin{enumerate} \item
Assume that every irreducible supercuspidal representation of a given reductive group $H$ defined over a $p$-adic field $F$
is a subrepresentation of the induction of a cuspidal
representation of a subgroup that is open compact modulo the center. Then $H$ satisfies property (PSC).
\item Let $\grp{H}$ be a reductive group defined over a number field $k$.
Then there exists a finite set $S_0$ of non-archimedean places of $k$ such that for all $v \notin S_0$ the group $\grp{H} (k_v)$ satisfies property (PSC).
Moreover, we can take the constant $c$ appearing in the definition of property (PSC) to be independent of $v \notin S_0$.
\end{enumerate}
\end{theorem}

As a consequence, we obtain the following global supplement to Theorem \ref{TheoremBD}.

\begin{theorem} \label{TheoremGlobal}
Let $G$ be a reductive group defined over a number field $k$.
\begin{enumerate} \item
There exists a finite set $S_0$ of non-archimedean places of $k$ such that $G$ satisfies property (BD) for the set $S_{\operatorname{fin}} - S_0$.
\item Suppose that the local groups $G(k_v)$ satisfy the local property (BD) for all $v$ in a set $T$ of non-archimedean places of $k$.
Then $G$ satisfies property (BD) for $T$.
\end{enumerate}
\end{theorem}

\begin{proof} In view of the second part of Theorem \ref{TheoremPSC},
it only remains to show that the constant $C$ in Theorem \ref{TheoremBD}
can be bounded by a uniform value for all $v \in S_{\operatorname{fin}} - S_0$.
This constant is bounded in terms of the constant in Theorem \ref{TheoremSupercuspidalBD},
which by Remark \ref{RemarkConstant} can be bounded in terms of the constant $C_5$.
Since Theorem \ref{TheoremPSC} gives also the boundedness of the constant in the definition of property (PSC),
we are reduced to bounding $C_1$, $C_2$ and $C_3$ uniformly in $v$.
We leave it to the reader to check that the latter constants can indeed be taken to be independent of $v$.
\end{proof}

Using known results on supercuspidal representations, we can be more precise in some cases.

\begin{corollary} \label{CorClassical}
\begin{enumerate} \item
Let $G$ be a split group of rank two or an inner form of $\GL(n)$ or $\SL(n)$ defined over a number field $k$.
Then $G$ satisfies property (BD) with respect to the set $S_{\operatorname{fin}}$ of all non-archimedean places.
\item Let $G$ be a symplectic, special orthogonal or unitary group defined over a number field $k$.
Then $G$ satisfies property (BD) with respect to the set
$S_{\operatorname{fin}} - \{ v  \in S_{\operatorname{fin}} : v | 2 \}$.
\end{enumerate}
\end{corollary}

\begin{proof}
This is again a consequence of Theorems \ref{TheoremBD} and \ref{TheoremPSC}.
The case of inner forms of $\GL (n)$ follows from the construction of S\'echerre \cite{MR2081220, MR2188448, MR2216835}
and the exhaustion result of S\'echerre--Stevens \cite{MR2427423}, extending earlier results of Bushnell--Kutzko \cite{MR1204652}
and Corwin \cite{MR1079053} for the general linear group itself.
(See \cite{MR2427423} and \cite{MR1204652} for a more complete history of the problem.) We can reduce the case of inner forms of $\SL (n)$ to this case since property (PSC) depends only on the derived group.
The case of classical groups follows from the result of Stevens \cite{MR2390287} (see also \cite[Appendix A]{MR3157998}).
\end{proof}

\begin{remark} \label{RemarkBD}
In \cite[Definition 5.9]{MR3352530}, property (BD) for a reductive group $\grp{G}$ defined over a number field $k$ was defined in a slightly different way.
However, the current property (BD) for the set $S_{\operatorname{fin}}$ implies property (BD) of [ibid.], and more precisely,
our current property (BD) for a set $T$ of non-archimedean places of $k$ implies the old property restricted to places $v \in T$.
To see this, note first that [ibid., Remark 5.13] already takes care of the case of $K_{0,v} \cap \hat{M} (k_v)^+$-invariants (level one at $v$ with respect to
$\hat{M} (k_v)^+$, using the notation of [ibid.]). In the remaining cases observe that the normalized intertwining operators used in [ibid.]
differ from the unnormalized operators used in the current paper by rational functions of $q^{-s}$ whose degree is bounded in terms of $\grp{G}$ only \cite{MR999488}.
The difference between the level with respect to the simply connected cover of $\hat{M}$ and the level with respect to $\hat{M}$ itself can also be accounted for
by adjusting the constant $C$ in the definition of property (BD).\footnote{We remark that in the older definition only Levi subgroups that are defined over $k$
were considered, which suffices for the application to the trace formula and the limit multiplicity problem. The current formulation seems more natural.}
\end{remark}

We can use the results of \cite{1504.04795} to draw consequences for the limit multiplicity problem, for which we also refer to
\cite{MR3352530} for more details. In \cite[Definition 1.2]{1504.04795} we defined the limit multiplicity property for a family $\mathcal{K}$ of open compact subgroups
$K$ of $G (\mathbb{A}^S)$, where $G$ is a reductive group defined over a number field $k$ and $S$ a finite set of places of $k$, including the archimedean places.
We also refer to [ibid., Definition 1.3] for the definition of a non-degenerate family of open compact subgroups.

\begin{definition}
Let $T$ be a finite set of non-archimedean places of $k$ that is disjoint to $S$.
We say that a family $\mathcal{K}$ of open compact subgroups of $G (\mathbb{A}^S)$ has \emph{bounded level} at $T$, if for every
$v \in T$ there exists an integer $n_v$ with $K_{n_v,v} \subset K$ for all $K \in \mathcal{K}$.
\end{definition}

In view of Remark \ref{RemarkBD} above, a trivial modification of the argument of \cite{1504.04795} yields the following variant of [ibid., Theorem 1.4]
(which is the case $S_0 = \emptyset$).

\begin{theorem}
Let $G$ be a reductive group defined over a number field $k$.
Suppose that $G$ satisfies property (TWN) and property (BD) with respect to the set $S_{\operatorname{fin}} - S_0$.
Let $S$ be a finite set of places of $k$, including the archimedean places,
and $\mathbf{K}_0^S$ an open compact subgroup of $G (\mathbb{A}^S)$.
Then limit multiplicity holds for any non-degenerate family $\mathcal{K}$ of open subgroups of $\mathbf{K}_0^S$ that has bounded level at $S_0 - S$.
\end{theorem}

In particular, by combining Corollary \ref{CorClassical} with the results of \cite{1603.05475} on property (TWN),
we obtain the following new examples of the limit multiplicity property.

\begin{corollary} \label{CorLM}
\begin{enumerate} \item
Let $G$ be a split group of rank two or an inner form of $\GL(n)$ or $\SL(n)$ defined over a number field $k$. Then limit multiplicity holds for any
non-degenerate family $\mathcal{K}$ of open subgroups of a given
open compact subgroup
$\mathbf{K}_0^S$ of $G (\mathbb{A}^S)$.
\item Let $G$ be a quasi-split classical group defined over a number field $k$.
Then limit multiplicity holds for any non-degenerate family $\mathcal{K}$ of open subgroups of
$\mathbf{K}_0^S$ that has bounded level at $\{ v \notin S : v | 2 \}$.
\end{enumerate}
\end{corollary}


\def\cprime{$'$}

\end{document}